\RequirePackage{fix-cm}
\documentclass[smallextended]{svjour3} 
\smartqed 
\usepackage{graphicx}
\usepackage{amssymb}
\usepackage{bm,amsmath,algorithm}
\usepackage{url,booktabs}
\usepackage{algpseudocode}
\usepackage{enumerate}
\usepackage{subfig}
\usepackage{cases}
\usepackage{comment}
\usepackage{empheq}
\usepackage{multirow}
\usepackage{caption}
\DeclareMathOperator*{\argmin}{argmin}

\DeclareMathOperator*{\rank}{rank}

\DeclareMathOperator*{\tr}{tr}
\newtheorem{assumption}[theorem]{Assumption}

\begin{document}
\title{A Block Coordinate Descent Method for Sensor Network Localization}

\author{Mitsuhiro Nishijima\and Kazuhide Nakata}
\institute{Mitsuhiro Nishijima \at Department of Industrial Engineering and Economics, Tokyo Institute of Technology, Ookayama, Meguro-ku, Tokyo 152-8550, Japan \\
 \email{nishijima.m.ae@m.titech.ac.jp} 
 \and
 Kazuhide Nakata \at Department of Industrial Engineering and Economics, Tokyo Institute of Technology, Ookayama, Meguro-ku, Tokyo 152-8550, Japan \\
 \email{nakata.k.ac@m.titech.ac.jp}}
\date{Received: date / Accepted: date} 
\maketitle
\begin{abstract}
The problem of sensor network localization (SNL) can be formulated as a semidefinite programming problem with a rank constraint. We propose a new method for solving such SNL problems. We factorize a semidefinite matrix with the rank constraint into a product of two matrices via the Burer--Monteiro factorization. Then, we add the difference of the two matrices, with a penalty parameter, to the objective function, thereby reformulating SNL as an unconstrained multiconvex optimization problem, to which we apply the block coordinate descent method. In this paper, we also provide theoretical analyses of the proposed method and show that each subproblem that is solved sequentially by the block coordinate descent method can also be solved analytically, with the sequence generated by our proposed algorithm converging to a stationary point of the objective function. We also give a range of the penalty parameter for which the two matrices used in the factorization agree at any accumulation point. Numerical experiments confirm that the proposed method does inherit the rank constraint and that it estimates sensor positions faster than other methods without sacrificing the estimation accuracy, especially when the measured distances contain errors.
\end{abstract}

\section{Preliminaries}
\label{sec:introduction}
\subsection{Introduction}
\label{subsec:background}
Sensor network localization (SNL) is the problem of estimating the unknown positions of $m$ sensors from the known positions of $n$ anchors and the measured distances (which may contain measurement errors) between sensor--sensor or sensor--anchor pairs. In this problem, let $E_{\mathrm{ss}}$ and $E_{\mathrm{sa}}$ be the sets of sensor--sensor and sensor--anchor pairs, respectively, with known measured distances, and let $d_{ij}$ and $d_{ik}$ be the measured distances for each $ij \coloneqq \{i,j\}\in E_{\mathrm{ss}}$ and $ik\coloneqq \{i,k\} \in E_{\mathrm{sa}}$, respectively. Let the anchor coordinates be $\bm{a}_k= (a_{k1},\dots,a_{kd})^\top\in\mathbb{R}^d \ (k = m + 1,\dots, m + n)$. Then, the SNL problem is formulated as the following system of equations with variables $\bm{x}_i\in\mathbb{R}^d \ (i = 1,\dots, m)$:
\begin{equation}
\|\bm{x}_i-\bm{x}_j\|_2 = d_{ij}\ (\forall ij \in E_{\mathrm{ss}}),\quad\|\bm{x}_i-\bm{a}_k\|_2 = d_{ik}\ (\forall ik \in E_{\mathrm{sa}}). \label{prob:original}
\end{equation}
When a matrix variable $Z$ is introduced, finding the $\bm{x}_1,\dots,\bm{x}_m$ satisfying system~(\ref{prob:original}) is known to be equivalent to finding a solution of the following semidefinite programming (SDP) problem with a rank constraint \cite{Biswas2004,Wan2019}:
\begin{equation}
\vline\quad\begin{aligned}
&\text{min}&&0\\
&\text{s.t.}&&A_{ij}\bullet Z = d_{ij}^2\quad(\forall ij\in E_{\mathrm{ss}}),\\
&&&A_{ik}\bullet Z = d_{ik}^2\quad(\forall ik\in E_{\mathrm{sa}}),\\
&&&Z_{(1:d,1:d)} = I_d,\\
&&&\rank(Z) \le d,\\
&&&Z \in \mathcal{S}_+^{d+m}.
\end{aligned}
\label{prob:rankSDP_noerror}
\end{equation}
Here, for each $ij\in E_{\rm ss}$ and $ik \in E_{\rm sa}$,
\begin{align*}
A_{ij}&\coloneqq\begin{pmatrix}
\bm{0}_d\\
\bm{e}_i-\bm{e}_j
\end{pmatrix}\begin{pmatrix}
\bm{0}_d\\
\bm{e}_i-\bm{e}_j
\end{pmatrix}^\top,\ \quad A_{ik} \coloneqq \begin{pmatrix}
\bm{a}_k\\
-\bm{e}_i
\end{pmatrix}\begin{pmatrix}
\bm{a}_k\\
-\bm{e}_i
\end{pmatrix}^\top,
\end{align*}
where $\bm{e}_1,\dots,\bm{e}_m$ is the canonical basis of $\mathbb{R}^m$. See Subsection~\ref{subsec:notation} for the definitions of the other symbols. When $d_{ij}$ and $d_{ik}$ contain measurement errors, problem~(\ref{prob:rankSDP_noerror}) generally does not have a solution. Therefore, to account for the case in which $d_{ij}$ and $d_{ik}$ contain errors, researchers have often considered problem~(\ref{prob:rankSDP_error}) defined below \cite{Biswas2004,Chang2017,Chang2014,Chang2016,Nie2009,Tseng2007}. In this problem, the distance constraints are incorporated in the objective function in the form of quadratic errors, and a penalty is imposed for violating those constraints. In this paper, we also seek to estimate sensor positions in SNL by solving this problem:
\begin{equation}
\vline\quad\begin{aligned}
&\text{min}&&\frac{1}{2}\sum_{ij\in E_{\mathrm{ss}}}(A_{ij}\bullet Z-d_{ij}^2)^2 + \frac{1}{2}\sum_{ik\in E_{\mathrm{sa}}}(
A_{ik}\bullet Z-d_{ik}^2)^2\\
&\text{s.t.}&&Z_{(1:d,1:d)}=I_d,\\
&&&\rank(Z) \le d,\\
&&&Z\in \mathcal{S}_+^{d+m}.
\end{aligned} \label{prob:rankSDP_error}
\end{equation}

The SNL problem is generally known to be NP-hard \cite{Aspnes2004}, and the formulations of the SNL problem as optimization problems (\ref{prob:rankSDP_noerror}) and (\ref{prob:rankSDP_error}) are also nonconvex. This nonconvexity is due to the rank constraint that appears in problems (\ref{prob:rankSDP_noerror}) and (\ref{prob:rankSDP_error}).
Therefore, many previous SNL studies removed the rank constraint and relaxed the problem into an SDP problem to estimate approximate sensor positions \cite{Biswas2006c,Biswas2006b,Biswas2004,Biswas2006,Liang2004,Wang2008}.
Among these methods, sparse full SDP (SFSDP) as proposed by Kim \textit{et al.} \cite{Kim2009} is an especially representative one. However, a solution of the SDP relaxation problem is not always a solution of the original problem. So and Ye \cite{So2007} referred to problem~(\ref{prob:original}), which has a unique solution and does not have sensor positions satisfying all of the given distances in a higher-dimensional space, as ``uniquely localizable." They proved that problem~(\ref{prob:original}) is uniquely localizable if and only if the maximum rank of the solutions of the SDP relaxation problem~(\ref{prob:rankSDP_noerror}) is $d$. Therefore, when the SDP problem is solved by the interior-point method, which is a representative method for solving general SDP problems, if the problem is uniquely localizable, then the exact sensor positions can be determined by solving the SDP relaxation problem. Otherwise, an optimal solution of the SDP relaxation problem corresponds to a configuration of the sensors in a higher-dimensional space because of the max-rank property of the interior-point method \cite{Goldfarb1998}, and it thus might give poorly estimated sensor positions in $d$-dimensional space.

On the other hand, the exact sensor positions in $d$-dimensional space can be estimated only if problem~(\ref{prob:original}) has a unique solution (Wan \textit{et al.} \cite{Wan2019} referred to a problem satisfying this condition as ``locatable"). Therefore, methods have recently emerged for estimating the sensor positions by solving problem (\ref{prob:rankSDP_noerror}) or (\ref{prob:rankSDP_error}) itself. Wan \textit{et al.} \cite{Wan2019} proposed a method that obtains a solution of problem~(\ref{prob:rankSDP_noerror}) by solving SDP problems multiple times. This approach is based on the fact that the rank of a matrix being less than or equal to $d$ is equivalent to its $(d+1)$th and subsequent eigenvalues all being zero. Numerical experiments showed that their method's estimation accuracy was better than that of an SDP relaxation-based method. However, their method took more time than the latter method, because it requires solving the SDP problem via an SDP solver at each iteration. Wan \textit{et al.} \cite{Wan2020} also proposed a method that transforms the SDP problem with the rank constraint into an SDP problem with a complementarity constraint and alternately performs minimization with regard to the two semidefinite matrices that appear in the problem. Numerical experiments also confirmed that this method was more accurate than SDP relaxation methods such as SFSDP. However, as in \cite{Wan2019}, it took too long to estimate the sensor positions.

Another method for solving general SDP problems is the Burer--Monteiro factorization, in which a semidefinite matrix $Z$ is factorized into the form $VV^\top$ and a nonconvex optimization problem is solved after the factorization \cite{Burer2003}. If the number of columns of $V$ is chosen as $r$ in this factorization, then it introduces the constraint that the rank must be less than or equal to $r$. Therefore, this method is suitable for obtaining low-rank solutions of SDP problems. In a series of studies, Chang and colleagues \cite{Chang2017,Chang2014,Chang2016} attempted to estimate sensor positions by using the Burer--Monteiro factorization. First, Chang and Xue \cite{Chang2014} proposed a method that applies the limited-memory Broyden--Fletcher--Goldfarb--Shanno method to the problem after the Burer--Monteiro factorization. However, they set the number of columns of $V$ used in the factorization to that of the semidefinite matrix before the factorization, so their method does not consider the rank constraint. Second, Chang \textit{et al.} \cite{Chang2016} used the same method as in \cite{Chang2014} to estimate sensor positions in three-dimensional space, but unlike in \cite{Chang2014}, they set the number of columns of $V$ to three, so we can say that this method takes the rank constraint into account. They compared it with SFSDP through numerical experiments and reported that the sensor positions could be estimated more quickly and with the same level of accuracy as SFSDP. However, those experiments involved only small problems with up to 200 sensors. Finally, Chang and Liu \cite{Chang2017} proposed a method that they called NLP-FD, which solves the optimization problem obtained from the Burer--Monteiro factorization by the curvilinear search algorithm \cite{Wen2013}. Their numerical experiments showed the superiority of NLP-FD over SFSDP when a problem is large in scale and the measured distances include errors.

In this paper, we propose a new method for SNL that accounts for the rank constraint. First, we factorize $Z$ in problem~(\ref{prob:rankSDP_error}) into a product of two matrices through the Burer--Monteiro factorization:
\begin{equation*}
Z = \begin{pmatrix}
I_d\\
U^\top
\end{pmatrix}\begin{pmatrix}
I_d\\
V^\top
\end{pmatrix}^\top.
\end{equation*}
This factorization is equivalent under the constraint $U - V =\bm{O}$. Therefore, problem~(\ref{prob:rankSDP_error}) can be transformed into an unconstrained multiconvex optimization problem by adding the difference between the two matrices, with a penalty parameter $\gamma$, to the objective function. Then, the block coordinate descent method can be applied to the new objective function, and optimization can be performed sequentially for each column of $U$ and $V$. We formalize this procedure as Algorithm~\ref{alg:node-based}.

We also analyze the proposed method theoretically. First, we show that each subproblem in Algorithm~\ref{alg:node-based} is an unconstrained convex quadratic optimization problem and can be solved analytically (Theorem~\ref{thm:analytic2}). Second, we show that any accumulation point of the sequence generated by Algorithm~\ref{alg:node-based} is a stationary point of the objective function (Theorem~\ref{thm:stationary2}). Third, we give a range of $\gamma$ for which the two matrices $U$ and $V$ used in the factorization coincide at any accumulation point (Theorem~\ref{thm:gamma}). Finally, we explain the relationship between the objective function in the reformulated problem and the augmented Lagrangian. Numerical experiments confirm that the proposed method does inherit the rank constraint; furthermore, the results demonstrate not only that our method estimates sensor positions faster than SFSDP and NLP-FD without sacrificing estimation accuracy, especially when the measured distances include errors, but also that our method does not run out of memory even for large-scale SNL problems.

The rest of this paper is organized as follows. In Section~\ref{sec:algorithm}, we present the proposed method and analyze it theoretically. In Section~\ref{sec:experiment}, we compare it with the other methods to confirm its effectiveness. Finally, in Section~\ref{sec:conclusion}, we present our conclusions and suggest possible future work.

\subsection{Notation}
\label{subsec:notation}
\begin{itemize}
\item $\mathbb{N}$ denotes the set of natural numbers without zero. $\mathbb{R}^{p}$ denotes the set of $p$-dimensional real vectors, and $\bm{0}_p$ denotes the zero vector of $\mathbb{R}^p$. When the size is clear from the context, we omit the size subscript at the lower right. $\mathbb{R}^{p\times q}$ denotes the set of $p\times q$ real matrices. Let $I_p$ be the identity matrix of $\mathbb{R}^{p\times p}$ and $\bm{O}$ be the zero matrix of an appropriate size. $\mathcal{S}_+^p$ denotes the set of $p\times p$ symmetric positive semidefinite matrices.
\item For $\bm{x} \in \mathbb{R}^p$, $\|\bm{x}\|_2$ denotes the $2$-norm of $\bm{x}$. For $A,\ B\in \mathbb{R}^{p\times q}$, $A\bullet B$ means the inner product between $A$ and $B$, denoted by $\tr(A^\top B)$; $\|A\|_F$ denotes the Frobenius norm of $A$; $A_{(i:j,k:l)}$ denotes the submatrix of $A$ obtained by choosing the $\{i,\dots,j\}$th rows of $A$ and the $\{k,\dots,l\}$th columns of $A$; and $\rank(A)$ denotes the rank of $A$. For any symmetric matrix $A$, $\lambda_{\mathrm{max}}(A)$ denotes its maximum eigenvalue.
\item For $ i=1,\dots,m$, $E_{\mathrm{ss}}[i]$ and $E_{\mathrm{sa}}[i]$ denote the sets of sensors and anchors, respectively, that are connected directly to sensor $i$.
\end{itemize}

\section{Proposed method and analyses}
\label{sec:algorithm}
\subsection{Proposed method}
\label{subsec:algorithm}
Problem~(\ref{prob:rankSDP_error}) is an SDP problem with a rank constraint and is difficult to solve directly. In this subsection, we propose a new method that transforms problem~(\ref{prob:rankSDP_error}) into an unconstrained multiconvex optimization problem and solves the latter problem sequentially to estimate sensor positions. First, a matrix $Z$ satisfies the three constraints in problem~(\ref{prob:rankSDP_error}) if and only if it can be factorized into the product of two matrices as follows:
\begin{equation*}
Z = \begin{pmatrix}
I_d& U\\
U^\top& U^\top U
\end{pmatrix} = \begin{pmatrix}
I_d\\
U^\top
\end{pmatrix}\begin{pmatrix}
I_d\\
V^\top
\end{pmatrix}^\top,\ U-V=\bm{O}. \label{eq:Burer-Monteiro2}
\end{equation*}
Thus, problem~(\ref{prob:rankSDP_error}) is equivalent to
\begin{equation}
\vline\quad\begin{aligned}
&\text{min}&&f(U,V) \coloneqq \frac{1}{2}\sum_{ij\in E_{\mathrm{ss}}}\left(A_{ij}\bullet \begin{pmatrix}
I_d\\
U^\top
\end{pmatrix}\begin{pmatrix}
I_d\\
V^\top
\end{pmatrix}^\top-d_{ij}^2\right)^2 \\
&&&+ \frac{1}{2}\sum_{ik\in E_{\mathrm{sa}}}\left(A_{ik}\bullet\begin{pmatrix}
I_d\\
U^\top
\end{pmatrix}\begin{pmatrix}
I_d\\
V^\top
\end{pmatrix}^\top-d_{ik}^2\right)^2\\
&\text{s.t.}&&U - V =\bm{O}.\\
\end{aligned} \label{prob:rankSDP_error3}
\end{equation}
To make problem~(\ref{prob:rankSDP_error3}) easier to solve, we remove the constraint $U - V=\bm{O}$ and add a quadratic penalty term $\gamma/2\|U-V\|_F$ with a penalty parameter $\gamma\ (>0)$ to the objective function; as a result, the objective function takes larger values as the constraint $U - V=\bm{O}$ is more strongly violated. In other words, we adopt the following unconstrained optimization problem:
\begin{equation}
\vline\quad
\begin{aligned}
&\text{min}&&F(U,V;\gamma) \coloneqq \frac{\gamma}{2}\|U-V\|_F^2 + f(U,V). \label{prob:biconvex2}
\end{aligned}
\end{equation}

In the proposed algorithm, we let $U = (\bm{u}_1,\dots,\bm{u}_m)$ and $V = (\bm{v}_1,\dots,\bm{v}_m)$ and then perform minimization with regard to $\bm{u}_1,\dots,\bm{u}_m,\bm{v}_1,\dots,\bm{v}_m$, i.e., each column of $U$ and $V$ sequentially. Specifically, the procedure is as listed in Algorithm~\ref{alg:node-based}.
\begin{algorithm}
\caption{Proposed algorithm for problem~(\ref{prob:biconvex2})}
\label{alg:node-based}
\begin{algorithmic}[1]
\Require an initial point $U^{(0)}=(\bm{u}_1^{(0)},\dots,\bm{u}_m^{(0)}),\ V^{(0)}=(\bm{v}_1^{(0)},\dots,\bm{v}_m^{(0)})\in\mathbb{R}^{d\times m}$,\ a penalty parameter $\gamma$, a parameter $\epsilon$
\Ensure a generated sequence $\{(U^{(p)}$,$V^{(p)})\}$
\While 1
\For {$i = 1,\dots,m$}
\State $\bm{u}_i^{(p)} = \argmin_{\bm{u}_i}F(\bm{u}_1^{(p)},\dots,\bm{u}_{i-1}^{(p)},\bm{u}_i,\bm{u}_{i+1}^{(p-1)},\dots,\bm{u}_m^{(p-1)},V^{(p-1)};\gamma)$.
\EndFor
\For {$i = 1,\dots,m$}
\State $\bm{v}_i^{(p)} = \argmin_{\bm{v}_i}F(U^{(p)},\bm{v}_1^{(p)},\dots,\bm{v}_{i-1}^{(p)},\bm{v}_i,\bm{v}_{i+1}^{(p-1)},\dots,\bm{v}_m^{(p-1)};\gamma)$.
\EndFor
\If{$\max\left\{\frac{2\|U^{(p)}-V^{(p)}\|_F}{\|U^{(p)}\|_F+\|V^{(p)}\|_F},\frac{\|U^{(p)}-U^{(p-1)}\|_F}{\|U^{(p-1)}\|_F},\frac{\|V^{(p)}-V^{(p-1)}\|_F}{\|V^{(p-1)}\|_F}\right\} < \epsilon$}
\State stop algorithm.
\Else
\State $p = p + 1$
\EndIf
\EndWhile
\end{algorithmic}
\end{algorithm}

The proposed method has the following advantages over other methods:
\begin{enumerate}[(i)]
\item The SDP problem with the rank constrained~(\ref{prob:rankSDP_error}) is equivalent to problem~(\ref{prob:rankSDP_error3}), from which we obtained problem~(\ref{prob:biconvex2}) by incorporating $U-V=\bm{O}$ in the objective function as a quadratic penalty term with a penalty parameter $\gamma$. As we will see from Theorem~\ref{thm:gamma}, if $\gamma$ is larger than a real-valued threshold, then the $U^{(p)}$ and $V^{(p)}$ generated by Algorithm~\ref{alg:node-based} coincide with each other at any accumulation point, thereby satisfying the constraint of problem~(\ref{prob:rankSDP_error3}). Therefore, the proposed method inherits the rank constraint in problem~(\ref{prob:rankSDP_error}) and retains the potential capability to estimate sensor positions accurately for problems that are not uniquely localizable. This advantage will be verified in Subsection~\ref{subsec:uniquely localizable}. \label{adv:1}
\item As we will see from Theorem~\ref{thm:analytic2}, each subproblem appearing inside a \textbf{for} statement in Algorithm~\ref{alg:node-based} is an unconstrained convex quadratic optimization problem. The solution of each subproblem can be obtained analytically, because the solution process can be reduced to solving a system of linear equations with an invertible coefficient matrix of size $d$. Because $d$ is at most three in real situations, the system can be solved rapidly and without running out of memory, regardless of the number of sensors $m$. Moreover, the subproblems only need to be solved $2m$ times (i.e., a number proportional to $m$) for each outer loop. Therefore, especially in the case of large-scale SNL problems, we expect faster estimates of the sensor positions as compared with other methods. This advantage will be verified in Subsection~\ref{subsec:2dim}. \label{adv:2}
\end{enumerate}

\subsection{Analyses of the proposed method}
\label{subsec:analysis}
In this subsection, we present theoretical analyses of problem~(\ref{prob:biconvex2}) and Algorithm~\ref{alg:node-based}. First, we impose an assumption about the problem that we are examining.
\begin{assumption}\label{asm:100}
All sensors are connected to an anchor either directly or indirectly.
\end{assumption}
The same assumption was also made in \cite{Chang2017,So2007,Wang2008} and is very natural when estimating sensor positions: if a sensor is not connected to any anchors, then its absolute position cannot be determined uniquely.

First, we prove that the optimal solution of each subproblem in Algorithm~\ref{alg:node-based} can be obtained uniquely as an analytical solution.
\begin{theorem}\label{thm:analytic2}
Fix $U'=(\bm{u}'_1,\dots,\bm{u}'_m)$ and $V'=(\bm{v}'_1,\dots,\bm{v}'_m)$ arbitrarily. Then, for each $i=1,\dots,m$, the solutions $\bm{u}_i^*,\ \bm{v}_i^*$ of the following two optimization problems
\begin{gather*}
\min_{\bm{u}_i\in\mathbb{R}^d}F(\bm{u}'_1,\dots,\bm{u}'_{i-1},\bm{u}_i,\bm{u}'_{i+1},\dots,\bm{u}'_m,V';\gamma),\\
\min_{\bm{v}_i\in\mathbb{R}^d}F(U',\bm{v}'_1,\dots,\bm{v}'_{i-1},\bm{v}_i,\bm{v}'_{i+1},\dots,\bm{v}'_m;\gamma)
\end{gather*}
are respectively $\bm{u}_i^* = A_{\bm{u}_i}^{-1}\bm{b}_{\bm{u}_i}$, $\bm{v}_i^* = A_{\bm{v}_i}^{-1}\bm{b}_{\bm{v}_i}$, where
\begin{align*}
A_{\bm{u}_i} &\coloneqq \gamma I_d + \sum_{j\in E_{\mathrm{ss}}[i]}(\bm{v}'_i-\bm{v}'_j)(\bm{v}'_i-\bm{v}'_j)^\top + \sum_{k\in E_{\mathrm{sa}}[i]}(\bm{v}'_i-\bm{a}_k)(\bm{v}'_i-\bm{a}_k)^\top ,\\
\bm{b}_{\bm{u}_i} &\coloneqq \gamma\bm{v}'_i + \sum_{j\in E_{\mathrm{ss}}[i]}((\bm{u}'_j)^\top\bm{v}'_i - (\bm{u}'_j)^\top\bm{v}'_j+d_{ij}^2)(\bm{v}'_i-\bm{v}'_j) \\
&+ \sum_{k\in E_{\mathrm{sa}}[i]}(\bm{a}_k^\top\bm{v}'_i-\bm{a}_k^\top\bm{a}_k+d_{ik}^2)(\bm{v}'_i-\bm{a}_k),\\
A_{\bm{v}_i} &\coloneqq \gamma I_d + \sum_{j\in E_{\mathrm{ss}}[i]}(\bm{u}'_i-\bm{u}'_j)(\bm{u}'_i-\bm{u}'_j)^\top + \sum_{k\in E_{\mathrm{sa}}[i]}(\bm{u}'_i-\bm{a}_k)(\bm{u}'_i-\bm{a}_k)^\top ,\\
\bm{b}_{\bm{v}_i} &\coloneqq \gamma\bm{u}'_i + \sum_{j\in E_{\mathrm{ss}}[i]}( (\bm{v}'_j)^\top\bm{u}'_i - (\bm{v}'_j)^\top\bm{u}'_j+d_{ij}^2)(\bm{u}'_i-\bm{u}'_j) \\
&+ \sum_{k\in E_{\mathrm{sa}}[i]}(\bm{a}_k^\top\bm{u}'_i-\bm{a}_k^\top\bm{a}_k+d_{ik}^2)(\bm{u}'_i-\bm{a}_k).\\
\end{align*}
\end{theorem}
\begin{proof}
If we focus on only $\bm{u}_i$ in $F (U, V;\gamma)$ in particular, then we can represent $F(\bm{u}'_1,\dots,\bm{u}'_{i-1},\bm{u}_i,\bm{u}'_{i+1},\dots,\bm{u}'_m,V';\gamma)$ as
\begin{align}
&F(\bm{u}'_1,\dots,\bm{u}'_{i-1},\bm{u}_i,\bm{u}'_{i+1},\dots,\bm{u}'_m,V';\gamma) \nonumber\\
&\quad= \frac{1}{2}\bm{u}_i^\top A_{\bm{u}_i}\bm{u}_i -\bm{b}_{\bm{u}_i}^\top \bm{u}_i + [\text{a constant unrelated to $\bm{u}_i$}]. \label{eq:Fu_i}
\end{align}
From equation~(\ref{eq:Fu_i}), we can see that $F(\bm{u}'_1,\dots,\bm{u}'_{i-1},\bm{u}_i,\bm{u}'_{i+1},\dots,\bm{u}'_m,V';\gamma)$ is $\gamma$-strongly convex. Thus, the optimal solution of
$$\min_{\bm{u}_i\in\mathbb{R}^d}F(\bm{u}'_1,\dots,\bm{u}'_{i-1},\bm{u}_i,\bm{u}'_{i+1},\dots,\bm{u}'_m,V';\gamma)$$
is the stationary point of $F(\bm{u}'_1,\dots,\bm{u}'_{i-1},\bm{u}_i,\bm{u}'_{i+1},\dots,\bm{u}'_m,V';\gamma)$. Because
\begin{equation*}
\nabla_{\bm{u}_i}F(\bm{u}'_1,\dots,\bm{u}'_{i-1},\bm{u}_i,\bm{u}'_{i+1},\dots,\bm{u}'_m,V';\gamma) = A_{\bm{u}_i}\bm{u}_i - \bm{b}_{\bm{u}_i}
\end{equation*}
and $A_{\bm{u}_i}$ is positive definite and invertible in particular, we can obtain $\bm{u}_i^*$. We can also obtain $\bm{v}_i^*$ from the same calculation. \qed
\end{proof}

Next, we show that the sequence generated by Algorithm~\ref{alg:node-based} converges to a stationary point of the objective function $F$.
\begin{theorem}\label{thm:stationary2}
Fix the penalty parameter $\gamma$ in problem~(\ref{prob:biconvex2}) arbitrarily. Let $\mathcal{N}$ be the set of stationary points of $F$. Then, the sequence $\{(U^{(p)},V^{(p)})\}_{p=1}^\infty$ generated by Algorithm~\ref{alg:node-based} satisfies
\begin{equation}
\lim_{p\to\infty}\inf_{(U,V)\in\mathcal{N}}\|(U^{(p)},V^{(p)})-(U,V)\|_F = 0. \label{eq:converge to Nash}
\end{equation}
In particular, any accumulation point $(U^*,V^*)$ of the generated sequence is a stationary point of $F$.
\end{theorem}
The consequence of Theorem~\ref{thm:stationary2} is based on the result of \cite{Xu2013}. By Corollary~2.4 in \cite{Xu2013}, if all three of the following conditions are satisfied, then equation~(\ref{eq:converge to Nash}) holds when the \textit{stationary point} of $F$ in the definition of $\mathcal{N}$ is replaced by the \textit{Nash equilibrium} of $F$ (see Definition~\ref{def:Nash} below).
\begin{enumerate}[\text{Condition}~(a):]
\item $F$ is continuous, bounded below, and has a Nash equilibrium.\label{condition:F}
\item The objective function of each subproblem is strongly convex.\footnote{To be more precise, a stronger assumption about the parameter of the strongly convex function is needed. However, in the present problem, the stronger assumption is satisfied automatically because the parameter is a constant $\gamma$. See Assumption~2 in \cite{Xu2013} for details.}\label{condition:strongly convex}
\item The sequence generated by Algorithm~\ref{alg:node-based} is bounded. \label{condition:bound}
\end{enumerate}
In the following, we prove Theorem~\ref{thm:stationary2} by showing the equivalence between the Nash equilibrium and the stationary point of $F$ and then verifying that all three conditions are satisfied.

We can see from equation~(\ref{eq:Fu_i}) that the function $F(U, V;\gamma)$ is convex on $\mathbb{R}^d$ if we focus only on each column of $U$ and $V$. Such a function $F$ with this property is called multiconvex~\cite{Xu2013}. For simplicity, let $\mathcal{X} \coloneqq \mathbb{R}^{n_1}\times\cdots\times\mathbb{R}^{n_s}$ (where $n_1,\dots,n_s \in \mathbb{N}$),\footnote{In the problem, we are considering the case in which $s = 2m$ and $n_i \equiv d$, so the simplification of $\mathcal{X}$ does not affect the discussion in this paper.} and when $\bm{x} \in \mathcal{X}$ is represented as $\bm{x} = (\bm{x}_1,\dots,\bm{x}_s)$, then $\bm{x}_i \in \mathbb{R}^{n_i}\ (i=1,\dots,s)$.
\begin{definition}\label{def:multiconvex}
A function $g: \mathcal{X}\to\mathbb{R}$ is called multiconvex on $\mathcal{X}$ (with respect to the block division $\bm{x} = (\bm{x}_1,\dots,\bm{x}_s)\in\mathcal{X}$) if for all $i=1,\dots,s$ and all $\bm{x}_j\in\mathbb{R}^{n_j}\ (j=1,\dots,i-1,i+1,\dots,m)$, the function
\begin{equation*}
g(\bm{x}_1,\dots,\bm{x}_{i-1},\bullet,\bm{x}_{i+1},\dots,\bm{x}_{s}):\ \mathbb{R}^{n_i}\to\mathbb{R}
\end{equation*}
is convex.
\end{definition}
One of the concepts of minimality for a multiconvex function is the Nash equilibrium~\cite{Xu2013}, which appears in Condition~(\ref{condition:F}).
\begin{definition}\label{def:Nash}
For a function $g:\mathcal{X}\to\mathbb{R}$, $(\bm{x}_1^*,\dots,\bm{x}_s^*)\in\mathcal{X}$ is called a Nash equilibrium of $g$ (with respect to the block division as in Definition~\ref{def:multiconvex}) if
\begin{equation*}
g(\bm{x}_1^*,\dots,\bm{x}_{i-1}^*,\bm{x}_i^*,\bm{x}_{i+1}^*,\dots,\bm{x}_s^*) \le g(\bm{x}_1^*,\dots,\bm{x}_{i-1}^*,\bm{x}_i,\bm{x}_{i+1}^*,\dots,\bm{x}_s^*)
\end{equation*}
holds for all $i=1,\dots,s$ and all $\bm{x}_i\in\mathbb{R}^{n_i}$.
\end{definition}
Gorski \textit{et al.} \cite{Gorski2007} proved the equivalence between the stationary point and the Nash equilibrium in the case of $s = 2$.\footnote{In \cite{Gorski2007}, the term ``partial optimum" is used instead of ``Nash equilibrium.''} Herein, we extend the equivalence to the case of arbitrary $s$.
\begin{lemma}\label{lem:equivalence2}
Let $g:\mathcal{X}\to\mathbb{R}$ be once differentiable and multiconvex. Then, $\bm{x}^* =(\bm{x}_1^*,\dots,\bm{x}_s^*)\in \mathcal{X}$ is a stationary point of $g$ if and only if $\bm{x}^*$ is a Nash equilibrium of $g$.
\end{lemma}
\begin{proof}
We begin by proving the ``if" part. If we assume that $\bm{x}^*$ is a stationary point of $g$, then because
\begin{equation*}
g_i(\bm{x}_i) \coloneqq g(\bm{x}_1^*,\dots,\bm{x}_{i-1}^*,\bm{x}_i,\bm{x}_{i+1}^*,\dots,\bm{x}_s^*)
\end{equation*}
is convex on $\mathbb{R}^{n_i}$ for all $i=1,\dots,m$,
\begin{equation}
g_i(\bm{x}_i) \ge g_i(\bm{x}_i^*) + \nabla_{\bm{x}_i}g_i(\bm{x}_i^*)^\top(\bm{x}_i-\bm{x}_i^*) \label{eq:thm4.2.1}
\end{equation}
holds for all $\bm{x}_i \in \mathbb{R}^{n_i}$. Because $\nabla_{\bm{x}_i}g_i(\bm{x}_i^*) = \bm{0}$ follows from the assumption of $\bm{x}^*$ being a stationary point of $g$, we can say from inequality~(\ref{eq:thm4.2.1}) that $g_i(\bm{x}_i) \ge g_i(\bm{x}_i^*)$ for all $\bm{x}_i\in\mathbb{R}^{n_i}$. Because $i$ is arbitrary, we can conclude that $\bm{x}^*$ is a Nash equilibrium of $g$.

Next, we prove the ``only if" part. If we assume that $\bm{x}^*$ is a Nash equilibrium of $g$, then for each $i = 1,\dots,s$, $g_i(\bm{x}_i)$ attains its minimum value at $\bm{x}_i = \bm{x}_i^*$, from which we obtain $\nabla_{\bm{x}_i}g_i(\bm{x}_i^*)=\bm{0}$. Thus, $\bm{x}^*$ is a stationary point of $g$. \qed
\end{proof}
\begin{lemma}\label{lem:bound}
Suppose that Assumption~\ref{asm:100} holds. Then, for all $\alpha$, the level set $S_F(\alpha) \coloneqq \{(U,V)\mid F(U,V;\gamma) \le \alpha\}$ is bounded and closed.
\end{lemma}
Because similar (but not the same) results were already pointed out in \cite{Chang2016,So2007}, we omit the proof of Lemma~\ref{lem:bound} because of the page limit. Note that if Assumption~\ref{asm:100} does not hold, then $S_F(\alpha)$ is always not bounded.
\begin{corollary}\label{cor:bound2}
For any initial point $(U^{(0)},V^{(0)})$, the sequence $\{(U^{(p)},V^{(p)})\}_{p=1}^\infty$ generated by Algorithm~\ref{alg:node-based} is bounded.
\end{corollary}
\begin{proof}
Let $\alpha\coloneqq F(U^{(0)},V^{(0)};\gamma)$. Because each minimization subproblem in Algorithm~\ref{alg:node-based} is strictly optimized, we can say that $(U^{(p)},V^{(p)})\in\ S_F(\alpha)$ for all $p\in\mathbb{N}$, i.e., $\{(U^{(p)},V^{(p)})\}_{p=1}^\infty\subseteq S_F(\alpha)$, which is bounded from Lemma~\ref{lem:bound}. \qed
\end{proof}
We can now prove Theorem~\ref{thm:stationary2} on the basis of the above claims.
\begin{proof}
The continuity and below-boundedness of $F$ are evident from its definition. When combined with Lemma~\ref{lem:bound}, the global optimal solution of problem~(\ref{prob:biconvex2}) is guaranteed to exist, from which the existence of a Nash equilibrium can be proved. Therefore, we can say that Condition~(\ref{condition:F}) holds. In addition, from equation~(\ref{eq:Fu_i}), the objective function of each subproblem is $\gamma$-strongly convex, and thus Condition~(\ref{condition:strongly convex}) holds. Finally, Condition~(\ref{condition:bound}) is also satisfied from Corollary~\ref{cor:bound2}. Therefore, the conditions of Corollary~2.4 in \cite{Xu2013} are all satisfied, from which we can show equation~(\ref{eq:converge to Nash}) by using the equivalence between the stationary point and the Nash equilibrium that was shown in Lemma~\ref{lem:equivalence2}. Because $F$ is of class $C^1$, $\mathcal{N}$ is closed, from which we can easily show the last part of Theorem~\ref{thm:stationary2}.\qed
\end{proof}

Finally, we show that the $U^{(p)}$ and $V^{(p)}$ generated by Algorithm~\ref{alg:node-based} coincide with each other at any accumulation point if $\gamma$ is larger than a real-valued threshold, which does not generally hold in the quadratic penalty method.
\begin{theorem}\label{thm:gamma}
For any initial point $(U^{(0)},V^{(0)})$ such that $U^{(0)}=V^{(0)}$, if
\begin{equation}
\gamma > \frac{1}{2}\sqrt{2f(U^{(0)},V^{(0)})}\max_{1\le i\le m}\sqrt{4|E_{\mathrm{ss}}[i]|+|E_{\mathrm{sa}}[i]|},\label{eq:gamma}
\end{equation}
then any accumulation point $(U^*,V^*)$ of the sequence $\{(U^{(p)},V^{(p)})\}_{p=1}^\infty$ generated by Algorithm~\ref{alg:node-based} satisfies $U^*=V^*$.
\end{theorem}
\begin{proof}
For each $p\in \mathbb{N},\ ij\in E_{\rm ss}$, and $ik\in E_{\rm sa}$, let
\begin{align*}
\alpha_{ij}^{(p)} &\coloneqq A_{ij}\bullet \begin{pmatrix}
I_d\\
(U^{(p)})^\top
\end{pmatrix}\begin{pmatrix}
I_d\\
(V^{(p)})^\top
\end{pmatrix}^\top -d_{ij}^2 ,\\
\alpha_{ik}^{(p)} &\coloneqq A_{ik}\bullet\begin{pmatrix}
I_d\\
(U^{(p)})^\top
\end{pmatrix}\begin{pmatrix}
I_d\\
(V^{(p)})^\top
\end{pmatrix}^\top -d_{ik}^2 .
\end{align*}
Then, because the initial point satisfies $U^{(0)}=V^{(0)}$ and the value of the objective function $F$ decreases monotonically by Algorithm~\ref{alg:node-based}, we can conclude that for all $p\in\mathbb{N}$,
\begin{align}
f(U^{(0)},V^{(0)}) &= F(U^{(0)},V^{(0)};\gamma)\ge F(U^{(p)},V^{(p)};\gamma) \nonumber\\
&\ge \frac{1}{2}\sum_{ij\in E_{\mathrm{ss}}}(\alpha_{ij}^{(p)})^2 + \frac{1}{2}\sum_{ik\in E_{\mathrm{sa}}}(\alpha_{ik}^{(p)})^2. \label{eq:f_ineq}
\end{align}
By taking a subsequence, without loss of generality, we can assume that $\{(U^{(p)},V^{(p)})\}_{p=1}^\infty$ itself converges to $(U^*,V^*)$. Then, it follows from inequality~(\ref{eq:f_ineq}) that
\begin{equation}
\frac{1}{2}\sum_{ij\in E_{\mathrm{ss}}}(\alpha_{ij}^*)^2 + \frac{1}{2}\sum_{ik\in E_{\mathrm{sa}}}(\alpha_{ik}^*)^2 \le f(U^{(0)},V^{(0)}). \label{eq:alpha}
\end{equation}
By Theorem~\ref{thm:stationary2}, $(U^*,V^*)$ is a stationary point of $F$, from which we have
\begin{equation*}
\nabla_{(U,V)}F(U^*,V^*;\gamma) = \bm{O}.
\end{equation*}
Thus, let $U_l,\ V_l\in\mathbb{R}^m$ be the respective $l$th-column vectors of $U^\top,\ V^\top$ for each $l=1,\dots,d$; then, $\nabla_{U_l}F(U^*,V^*;\gamma) = \nabla_{V_l}F(U^*,V^*;\gamma) = \bm{0}$ holds. For each $ik\in E_{\mathrm{sa}}$ and $l=1,\dots,d$, let $\bm{b}_{ik}^l$ be an $m$-dimensional vector such that its $i$th component is $-a_{kl}$ and all other components are zeros. Furthermore, for each $ij\in E_{\mathrm{ss}}$ and $ik\in E_{\mathrm{sa}}$, let $\bar{A}_{ij}$ and $\bar{A}_{ik}$ respectively be
\begin{align*}
\bar{A}_{ij} \coloneqq (A_{ij})_{(d+1:d+m,d+1:d+m)},\ \bar{A}_{ik} &\coloneqq (A_{ik})_{(d+1:d+m,d+1:d+m)}.
\end{align*}
Using these symbols, we can represent $F(U,V;\gamma)$ as
\begin{align*}
F(U,V;\gamma) &= \frac{\gamma}{2}\sum_{l=1}^d\|U_l-V_l\|_2^2 + \frac{1}{2}\sum_{ij\in E_{\mathrm{ss}}}\left(\sum_{l=1}^dU_l^\top\bar{A}_{ij}V_l-d_{ij}^2\right)^2\\
&+ \frac{1}{2}\sum_{ik\in E_{\mathrm{sa}}}\left(\sum_{l=1}^dU_l^\top\bar{A}_{ik}V_l + \sum_{l=1}^d(\bm{b}_{ik}^l)^\top(U_l+V_l) + \bm{a}_k^\top\bm{a}_k -d_{ik}^2\right)^2.
\end{align*}
Because
\begin{align*}
\nabla_{U_l}F(U^*,V^*;\gamma) &= \gamma(U_l^*-V_l^*) + \sum_{ij\in E_{\mathrm{ss}}}\alpha_{ij}^*\bar{A}_{ij}V_l^* + \sum_{ik\in E_{\mathrm{sa}}}\alpha_{ik}^*(\bar{A}_{ik}V_l^*+\bm{b}_{ik}^l) \\
&= \bm{0},\\
\nabla_{V_l}F(U^*,V^*;\gamma) &= \gamma(V_l^*-U_l^*) + \sum_{ij\in E_{\mathrm{ss}}}\alpha_{ij}^*\bar{A}_{ij}U_l^* + \sum_{ik\in E_{\mathrm{sa}}}\alpha_{ik}^*(\bar{A}_{ik}U_l^*+\bm{b}_{ik}^l) \\
&= \bm{0}
\end{align*}
for all $l=1,\dots,d$, we obtain
\begin{align}
&(U_l^*-V_l^*)^\top\nabla_{U_l}F(U^*,V^*;\gamma) + (V_l^*-U_l^*)^\top\nabla_{V_l}F(U^*,V^*;\gamma) \nonumber\\
&=(U_l^*-V_l^*)^\top\left\{2\gamma I_m-\left(\sum_{ij\in E_{\mathrm{ss}}}\alpha_{ij}^*\bar{A}_{ij}+\sum_{ik\in E_{\mathrm{sa}}}\alpha_{ik}^*\bar{A}_{ik}\right)\right\}(U_l^*-V_l^*) = 0. \label{eq:positive_definite}
\end{align}
For convenience, let
\begin{equation*}
\bar{A} \coloneqq \sum_{ij\in E_{\mathrm{ss}}}\alpha_{ij}^*\bar{A}_{ij}+\sum_{ik\in E_{\mathrm{sa}}}\alpha_{ik}^*\bar{A}_{ik}.
\end{equation*}
Next, we seek to prove the following inequality:
\begin{equation}
\lambda_{\mathrm{max}}(\bar{A}) \le \sqrt{2f(U^{(0)},V^{(0)})}\max_{1\le i\le m}\sqrt{4|E_{\mathrm{ss}}[i]|+|E_{\mathrm{sa}}[i]|}. \label{eq:gamma2}
\end{equation}
In fact, if inequality~(\ref{eq:gamma2}) can be shown, then because $\gamma$ satisfies inequality~(\ref{eq:gamma}), $2\gamma I_m-\bar{A}$ is a positive definite matrix, and thus, $U_l^*=V_l^*$ from equality~(\ref{eq:positive_definite}). Because of the arbitrariness of $l$, we can eventually conclude that $U^* = V^*$. Therefore, we need only prove inequality~(\ref{eq:gamma2}).

It follows from the Gershgorin circle theorem that
\begin{equation}
\lambda_{\mathrm{max}}(\bar{A}) \le \max_{1\le i\le m}\left\{\sum_{j\in E_{\mathrm{ss}}[i]}\alpha_{ij}^* + \sum_{k\in E_{\mathrm{sa}}[i]}\alpha_{ik}^* + \sum_{j\in E_{\mathrm{ss}}[i]}|\alpha_{ij}^*|\right\}. \label{eq:Gershgorin}
\end{equation}
For each $i=1,\dots,m$, let $v(i)$ be the optimal value of the following optimization problem:\footnote{Although we use the notations $``ij"$ and $``ik"$ to denote the indices of the variables $\alpha$ in the sums in the constraint of this optimization problem, they are not related to $i$ ($=1,\dots,m$), which is fixed here.}
\begin{equation*}
\vline\quad\begin{aligned}
&\text{max}&&\sum_{j\in E_{\mathrm{ss}}[i]}\alpha_{ij} + \sum_{k\in E_{\mathrm{sa}}[i]}\alpha_{ik} + \sum_{j\in E_{\mathrm{ss}}[i]}|\alpha_{ij}|\\
&\text{s.t.}&&\frac{1}{2}\sum_{ij\in E_{\mathrm{ss}}}\alpha_{ij}^2 + \frac{1}{2}\sum_{ik\in E_{\mathrm{sa}}}\alpha_{ik}^2 \le f(U^{(0)},V^{(0)}),
\end{aligned}\label{prob:alpha}
\end{equation*}
where the right side of inequality~(\ref{eq:Gershgorin}) does not exceed $\max_{1\le i\le m}v(i)$ because of inequality~(\ref{eq:alpha}). We can easily check that $v(i)$ is equal to the optimal value of the following optimization problem for each $i=1,\dots,m$:
\begin{equation}
\vline\quad\begin{aligned}
&\text{max}&&2\sum_{j\in E_{\mathrm{ss}}[i]}\alpha_{ij} + \sum_{k\in E_{\mathrm{sa}}[i]}\alpha_{ik}\\
&\text{s.t.}&&\frac{1}{2}\sum_{j\in E_{\mathrm{ss}}[i]}\alpha_{ij}^2 + \frac{1}{2}\sum_{k\in E_{\mathrm{sa}}[i]}\alpha_{ik}^2=f(U^{(0)},V^{(0)}).
\end{aligned}\label{prob:alpha2}
\end{equation}
Using the method of Lagrange multipliers, we can see that the optimal value of problem~(\ref{prob:alpha2}) is $\sqrt{2f(U^{(0)},V^{(0)})}\sqrt{4|E_{\mathrm{ss}}[i]|+|E_{\mathrm{sa}}[i]|}$. Therefore,
\begin{align*}
&\max_{1\le i\le m}\left\{\sum_{j\in E_{\mathrm{ss}}[i]}\alpha_{ij}^* + \sum_{k\in E_{\mathrm{sa}}[i]}\alpha_{ik}^* + \sum_{j\in E_{\mathrm{ss}}[i]}|\alpha_{ij}^*|\right\} \\
&\quad\le \max_{1\le i\le m}v(i) = \sqrt{2f(U^{(0)},V^{(0)})}\max_{1\le i\le m}\sqrt{4|E_{\mathrm{ss}}[i]|+|E_{\mathrm{sa}}[i]|},
\end{align*}
which implies inequality~(\ref{eq:gamma2}). \qed
\end{proof}

\subsection{Relationship to the augmented Lagrangian}
\label{subsec:Augmented Lagrangian}
In this paper, we adopt the quadratic-penalty-based method, in which the equality constraint $U - V = \bm{O}$ is incorporated in the objective function as a quadratic penalty term and the resulting new objective function is minimized. On the other hand, there are also methods such as the augmented Lagrangian method~\cite{Nocedal2006} and the alternating direction method of multipliers~\cite{Boyd2010} that minimize the augmented Lagrangian, which contains not only the quadratic penalty term but also the Lagrange multiplier term. It is known that the augmented Lagrangian method is more efficient than the quadratic penalty method. For example, while the quadratic penalty method requires the penalty parameter to diverge to positive infinity, the augmented Lagrangian method does not require it to diverge, and the sequence obtained by the augmented Lagrangian method converges faster than that obtained by the quadratic penalty method~\cite[Example~17.4]{Nocedal2006}. Hence, we explain that problem~(\ref{prob:biconvex2}) can be regarded as a minimization problem of the augmented Lagrangian with an exact Lagrangian multiplier.

$\Lambda = \bm{O}$ is the exact Lagrange multiplier of problem~(\ref{prob:rankSDP_error3}). In fact, for all local optimum solutions $(U^*,V^*)$ of problem~(\ref{prob:rankSDP_error3}), because problem~(\ref{prob:rankSDP_error3}) satisfies the linear independence constraint qualification, there exists a Lagrange multiplier $\Lambda^*\in\mathbb{R}^{d\times m}$ satisfying the Karush--Kuhn--Tucker condition. In other words, if we let the Lagrangian for problem~(\ref{prob:rankSDP_error3}) be
\begin{equation*}
\mathcal{L}(U,V,\Lambda) \coloneqq f(U,V)-\Lambda\bullet(U-V),
\end{equation*}
then
\begin{subequations}
\begin{empheq}[left=\empheqlbrace]{align}
&\nabla_{U}\mathcal{L}(U^*,V^*,\Lambda^*) = \nabla_{U}f(U^*,V^*) - \Lambda^* = \bm{O}, \label{eq:nablaU}\\
&\nabla_{V}\mathcal{L}(U^*,V^*,\Lambda^*) = \nabla_{V}f(U^*,V^*) + \Lambda^* = \bm{O}, \label{eq:nablaV}\\
&\nabla_{\Lambda}\mathcal{L}(U^*,V^*,\Lambda^*) = -(U^* - V^*) = \bm{O} \label{eq:nablaLambda}
\end{empheq}
\end{subequations}
hold. We get $U^* = V^*$ from equation~(\ref{eq:nablaLambda}) and denote both of them as $W^*$. Because $f(U,V) = f(V,U)\ (\forall U,\ V\in\mathbb{R}^{d\times m})$, $\nabla_{U}f(W^*,W^*) = \nabla_{V}f(W^*,W^*)$. Using this equation and equations (\ref{eq:nablaU}) and (\ref{eq:nablaV}), we obtain $\Lambda^* = \bm{O}$. Therefore, the augmented Lagrangian with the Lagrange multiplier $\Lambda = \bm{O}$ and penalty parameter $\gamma$ is
\begin{equation*}
f(U,V) - \bm{O}\bullet (U-V) + \frac{\gamma}{2}\|U-V\|_F^2 = f(U,V) + \frac{\gamma}{2}\|U-V\|_F^2,
\end{equation*}
which is the definition of $F(U,V;\gamma)$ itself. Hence, problem~(\ref{prob:biconvex2}) can be regarded as a minimization problem of the augmented Lagrangian with the exact Lagrange multiplier $\Lambda = \bm{O}$ for problem~(\ref{prob:rankSDP_error3}).

\section{Numerical experiments}
\label{sec:experiment}
In this section, we use numerical simulation to verify the advantages (\ref{adv:1}) and (\ref{adv:2}) described in Subsection~\ref{subsec:algorithm} for the proposed method. We begin by confirming that our method does inherit the rank constraint; to confirm this, we compare it with an SDP relaxation-based method for a problem that is locatable but not uniquely localizable. Next, to confirm the effectiveness of the proposed method, we compare its estimation time and estimation accuracy with those of other methods by using artificial data under various conditions. All experiments were conducted on a computer with the macOS Catalina operating system, an Intel Core i5-8279U 2.40~GHz CPU, and 16~GB of memory. All the algorithms were implemented using MATLAB (R2020a). The parameter $\epsilon$ in Algorithm~\ref{alg:node-based} was set to $10^{-5}$ throughout the experiments.

\subsection{Comparison with SFSDP for a problem that is not uniquely localizable}
\label{subsec:uniquely localizable}
In this subsection, we demonstrate that the proposed method has the capability to estimate sensor positions accurately for a problem that is locatable but not uniquely localizable. Specifically, we examine a problem from \cite{So2007}:
\begin{gather*}
\|\bm{x}_1-\bm{x}_2\|_2 = \sqrt{10}/5,\ \|\bm{x}_1-\bm{a}_4\|_2 = \sqrt{5}/2,\ \|\bm{x}_1-\bm{a}_5\|_2 = \sqrt{5}/2,\\
\|\bm{x}_2-\bm{a}_3\|_2 = \sqrt{85}/10,\ \|\bm{x}_2-\bm{a}_5\|_2 = \sqrt{65}/10,
\end{gather*}
where $\bm{a}_3 = (0, 1.4)^\top$, $\bm{a}_4 = (-1, 0)^\top$, and $\bm{a}_5 = (1, 0)^\top$, and the true positions of the two sensors are $\bm{x}_1^{\rm true}=(0,0.5)^\top,\ \bm{x}_2^{\rm true}=(0.6,0.7)^\top$. For this problem, Algorithm~\ref{alg:node-based} was executed after fixing the penalty parameter $\gamma$ as $\sqrt{2f(U^{(0)},V^{(0)})}\max_{1\le i\le m}\sqrt{4|E_{\rm ss}| + |E_{\rm sa}|}/2$ according to Theorem~\ref{thm:gamma}. We examined the two cases of whether the initial points $\bm{u}_i^{(0)}\ (= \bm{v}_i^{(0)}) \in \mathbb{R}^2\ (i=1,2)$ in Algorithm~\ref{alg:node-based} are in the interior or the exterior of the convex hull of the three anchors. Figure~\ref{fig:Experiment6} shows the sensor positions estimated by SFSDP, the SDP relaxation-based method described in Section~\ref{sec:introduction}, and the proposed method. Note that when we estimated the sensor positions with the proposed method, the randomness of the initial points was varied 10 times. The results were similar to those of Figure~\ref{fig:in conv} in all cases in which the initial points were in the interior of the convex hull of the anchors. On the other hand, the results were similar to those of either Figure~\ref{fig:not in conv_5} or Figure~\ref{fig:not in conv_1} in all cases in which the initial points were in the exterior of the convex hull. Accordingly, only these three cases are included in Figure~\ref{fig:Experiment6}.
\begin{figure}[tb]
\centering
\subfloat[SFSDP]{\includegraphics[width=0.5\linewidth]{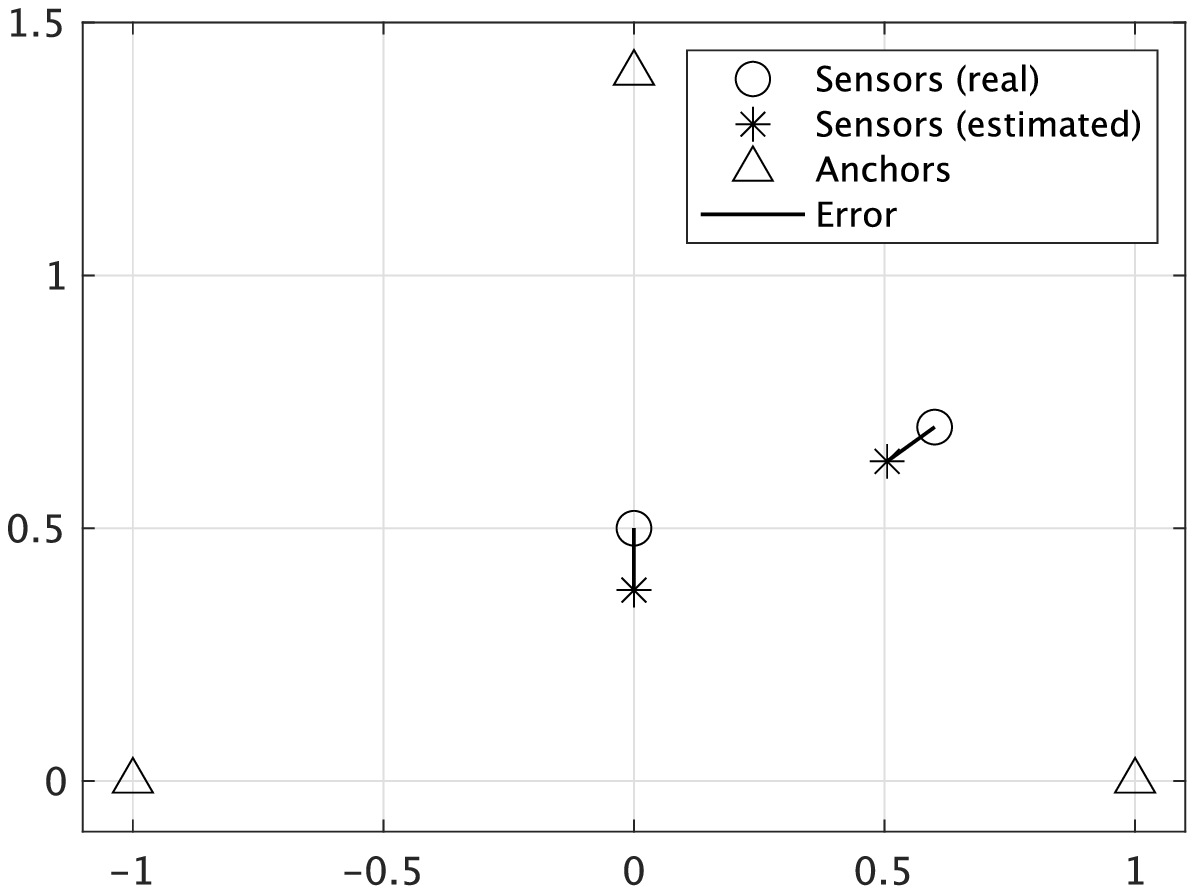}\label{fig:Experiment6_SFSDP}}
\subfloat[Initial points in the interior of the convex hull]{\includegraphics[width=0.5\linewidth]{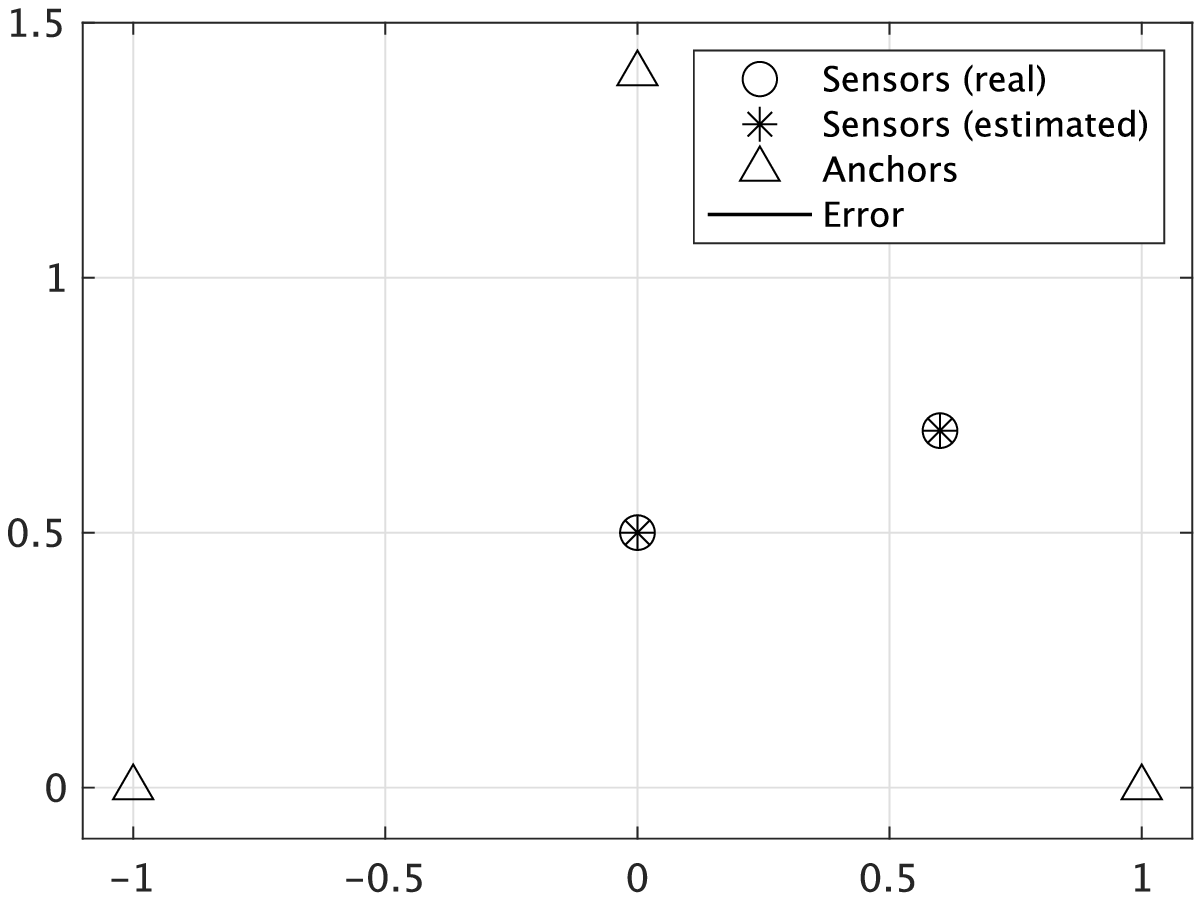}\label{fig:in conv}}\\
\subfloat[Initial points in the exterior of the convex hull (case 1)]{\includegraphics[width=0.5\linewidth]{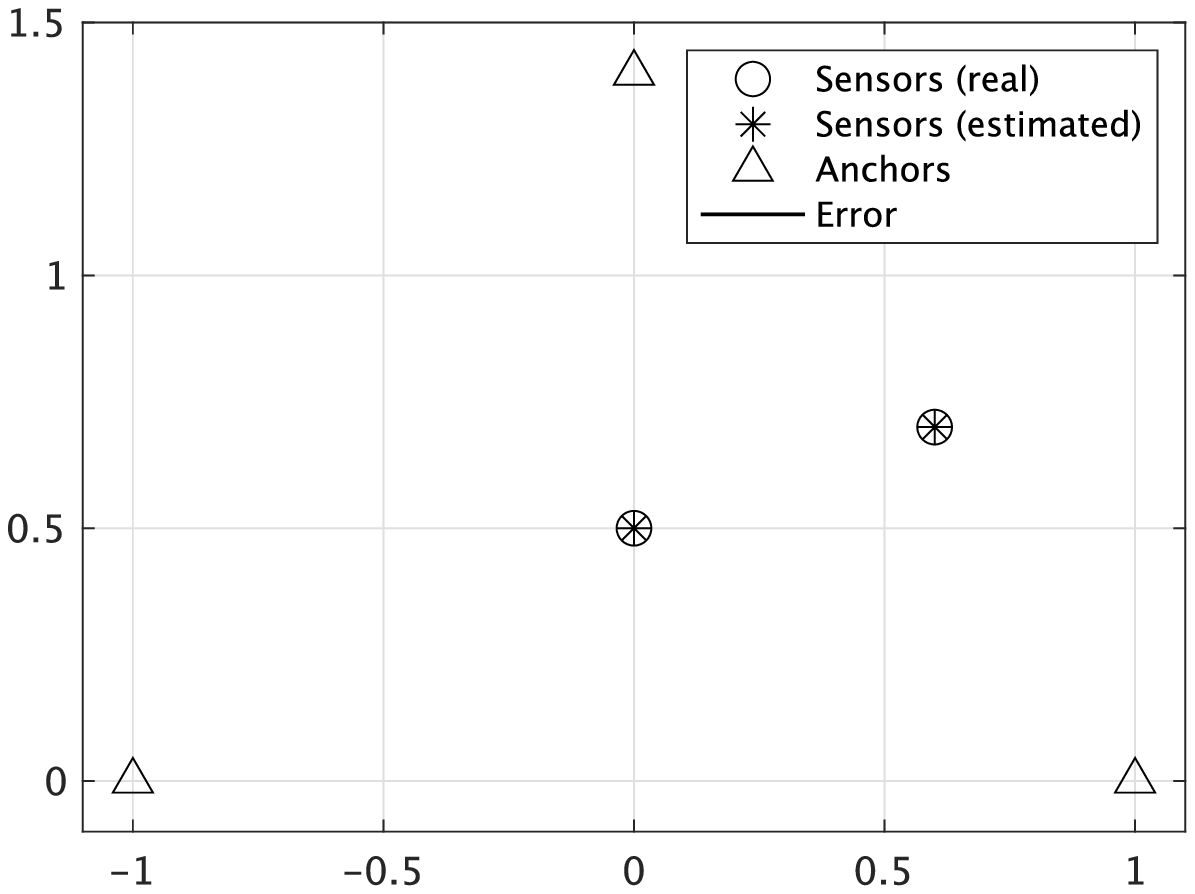}\label{fig:not in conv_5}}
\subfloat[Initial points in the exterior of the convex hull (case 2)]{\includegraphics[width=0.5\linewidth]{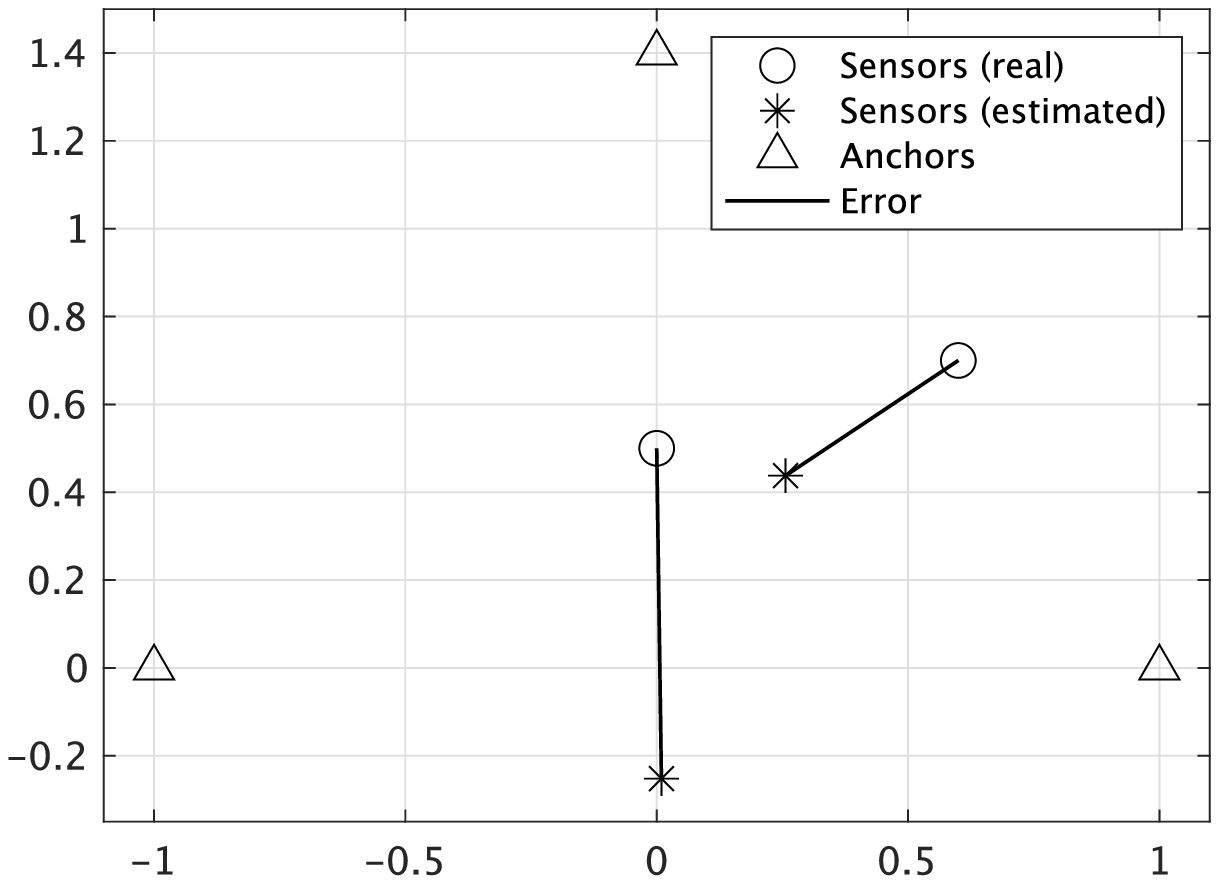}\label{fig:not in conv_1}}
\caption{Estimated sensor positions when the locatable but non-uniquely localizable problem was solved by \protect\subref{fig:Experiment6_SFSDP} SFSDP and \protect\subref{fig:in conv}--\protect\subref{fig:not in conv_1} the proposed method.}\label{fig:Experiment6}
\end{figure}

When we used SFSDP, the sensor positions were not estimated correctly (Figure~\ref{fig:Experiment6_SFSDP}). For the proposed method, the estimation accuracy depended on the initial points. When the initial points were in the interior of the convex hull of the anchors, the sensor positions were estimated accurately (Figure~\ref{fig:in conv}). On the other hand, when the initial points were in the exterior of the convex hull, the sensor positions were estimated accurately in some cases (Figure~\ref{fig:not in conv_5}) but not in others (Figure~\ref{fig:not in conv_1}).

Of course, because problem~(\ref{prob:biconvex2}) examined in this paper is a nonconvex optimization problem, whether the sensor positions can be estimated accurately depends on the initial points. However, as shown in Figure~\ref{fig:in conv} and Figure~\ref{fig:not in conv_5}, the proposed method still has the capability to estimate sensor positions accurately even for problems that are not uniquely localizable, although the example here is quite simple. On the other hand, when we use SDP relaxation-based methods, if a given problem is not uniquely localizable, there is no capability for accurate estimation because of the max-rank property of the interior-point method, as described in Section~\ref{sec:introduction}. Therefore, we can say that the proposed method does inherit the rank constraint.

\subsection{Comparison of estimation time and accuracy}
\label{subsec:2dim}
In this subsection, we quantitatively compare the estimation time and the estimation accuracy of the proposed method with those of existing methods for sensors located in two- or three- dimensional space. The compared methods are SFSDP, which was also used in Subsection~\ref{subsec:uniquely localizable}, and NLP-FD, which takes the rank constraint into account as our proposed method does. Although we introduced the methods proposed by Wan \textit{et al.} \cite{Wan2020,Wan2019} in Section~\ref{sec:introduction}, which also account for the rank constraint, we do not compare them here because of their extremely low scalability. In these experiments, $m = 1000,\ 3000,\ 5000$, and $20000$ sensors and $n = 0.1m$ anchors were placed randomly in $[0, 1]^d$. $E_{\mathrm{ss}}$ and $E_{\mathrm{sa}}$ were defined as
\begin{align*}
E_{\mathrm{ss}} &\coloneqq \{ij\mid 1\le i<j\le m,\ \|\bm{x}_i^{\mathrm{true}}-\bm{x}_j^{\mathrm{true}}\|_2 < \rho\},\\
E_{\mathrm{sa}} &\coloneqq \{ik\mid 1\le i\le m,\ m+1\le k\le m+n,\ \|\bm{x}_i^{\mathrm{true}}-\bm{a}_k\|_2 < \rho\},
\end{align*}
where the $\bm{x}_i^{\mathrm{true}}\ (i=1,...,m)$ are the sensors' true positions. In other words, we considered a model in which the distance between two sensors or between a sensor and an anchor is observed if and only if it is less than a radio range $\rho\ (>0)$. We set $\rho$ to $0.1$ and $\sqrt{10/m}$ in the case of $d=2$ and to $0.25$ and $\sqrt[3]{15/m}$ in the case of $d=3$. The measured distances $d_{ij}\ (ij\in E_{\mathrm{ss}})$ and $d_{ik}\ (ik\in E_{\mathrm{sa}})$ were given by
\begin{align*}
d_{ij} &= \max\{(1+\sigma \epsilon_{ij}),0.1\}\|\bm{x}_i^{\mathrm{true}}-\bm{x}_{j}^{\mathrm{true}}\|_2,\\
d_{ik} &= \max\{(1+\sigma \epsilon_{ik}),0.1\}\|\bm{x}_i^{\mathrm{true}}-\bm{a}_{k}\|_2,\\
\end{align*}
where $\epsilon_{ij},\ \epsilon_{ik}$ were selected independently from the standard normal distribution, and $\sigma$ is a noise factor determining the influence of the error. $\sigma$ was set to $0,\ 0.1$, and 0.2. As an indicator to measure the estimation accuracy, we used the root-mean-square distance (RMSD), which has been used in many other papers on SNL \cite{Chang2017,Kim2009,Wan2020,Wang2008} and is defined as
\begin{equation*}
\mathrm{RMSD} \coloneqq \sqrt{\frac{1}{m}\sum_{i=1}^m\|\hat{\bm{x}}_i-\bm{x}_i^{\mathrm{true}}\|_2^2},
\end{equation*}
where $\hat{\bm{x}}_i$ is the estimated position of sensor $i$. For each set of $(m,n,\rho,\sigma)$, five different problems of varying randomness were created, and the final results were the averages of five measurements of the estimation time (CPU time) and the estimation accuracy (RMSD).

The initial point $(U^{(0)},V^{(0)})\in\mathbb{R}^{d\times m}\times\mathbb{R}^{d\times m}$ in Algorithm~\ref{alg:node-based} was decided similarly to the method in \cite{Chang2017}. That is, for each sensor $i\ (i=1,\dots,m)$, if it was connected directly to an anchor, then $\bm{u}_i^{(0)}$ and $\bm{v}_i^{(0)}$ were set to the coordinates of the anchor nearest to sensor $i$; otherwise, $\bm{u}_i^{(0)}$ and $\bm{v}_i^{(0)}$ were set to
\begin{equation*}
\frac{1}{2}\left(\begin{pmatrix}
\max_{k}a_{k1}\\
\vdots\\
\max_{k}a_{kd}
\end{pmatrix} + \begin{pmatrix}
\min_{k}a_{k1}\\
\vdots\\
\min_{k}a_{kd}
\end{pmatrix}\right).
\end{equation*}
The penalty parameter $\gamma$ was updated dynamically according to Theorem~\ref{thm:gamma} by the following procedure.
\begin{enumerate}[Step~1.]
\item Let
$$\gamma^{(0)} = 5 \times 10^{-3} \times \sqrt{2f(U^{(0)},V^{(0)})}\max_{1\le i\le m}\sqrt{4|E_{\mathrm{ss}}[i]|+|E_{\mathrm{sa}}[i]|}/2.$$
By using $\gamma^{(0)}$ as the penalty parameter, $(U^{(1)},V^{(1)})$ is calculated by the update rule in the \textbf{while} loop of Algorithm~\ref{alg:node-based}. Let $\gamma^{(1)} = \gamma^{(0)}/2$. Then, by using $\gamma^{(1)}$ as the penalty parameter, $(U^{(2)},V^{(2)})$ is also calculated by this update rule. Let $p=2$. \label{step:oldstart}\\
\item If
\begin{align*}
&\frac{f(U^{(p-1)},V^{(p-1)})-f(U^{(p)},V^{(p)})}{f(U^{(p-1)},V^{(p-1)})}\\
&\quad\ge \frac{f(U^{(p-2)},V^{(p-2)})-f(U^{(p-1)},V^{(p-1)})}{f(U^{(p-2)},V^{(p-2)})},
\end{align*}
then $\gamma^{(p)} = (\gamma^{(p-1)}/\gamma^{(p-2)})\gamma^{(p-1)}$; otherwise, $\gamma^{(p)} = \gamma^{(p-2)}$. By using $\gamma^{(p)}$ as the penalty parameter, $(U^{(p+1)},V^{(p+1)})$ is calculated by the update rule in the \textbf{while} loop of Algorithm~\ref{alg:node-based}. \label{step:gammaold}\\
\item If $|f(U^{(p)},V^{(p)})-f(U^{(p+1)},V^{(p+1)})|/f(U^{(p)},V^{(p)}) < 10^{-2}$ or the overall stopping criterion (line 8 of Algorithm~\ref{alg:node-based}) is satisfied, then go to Step~\ref{step:iteration};\footnote{Even if the overall stopping criterion is satisfied at this stage, the entire algorithm does not end but always proceeds to Step~\ref{step:iteration}.} otherwise, set $p=p+1$ and go to Step~\ref{step:gammaold}. \label{step:oldend}\\
\item Let $W\coloneqq (U^{(p)}+V^{(p)})/2$ and
\begin{equation*}
\gamma = \sqrt{2f(W,W)}\max_{1\le i\le m}\sqrt{4|E_{\mathrm{ss}}[i]|+|E_{\mathrm{sa}}[i]|}/2.
\end{equation*}
Restart Algorithm~\ref{alg:node-based} with $W$ as the initial point and $\gamma$ as the penalty parameter. \label{step:iteration}
\end{enumerate}
The method of updating $\gamma$ described above consists of two components. First, in Steps~\ref{step:oldstart}--\ref{step:oldend}, $\gamma$ is updated so that the value of $f$, which represents the squared error of the squared distances, decreases rather than the penalty term $\|U-V\|_F$. However, if we keep reducing the value of $f$ rather than the penalty term, then the penalty term does not decrease much, and $U$ and $V$ may end up taking very different values from each other. Therefore, in Step~\ref{step:iteration}, we try to reduce the difference between $U$ and $V$ by fixing $\gamma$ according to Theorem~\ref{thm:gamma}.
\begin{table}[tb]
\centering
\caption{Results of numerical experiments for sensors and anchors placed randomly in $[0, 1]^2$. For each parameter combination, the lowest CPU time and shortest root-mean-square distance (RMSD) are in bold.}
\scalebox{0.89}{\begin{tabular}{|l|l|l||r|r|r|r|r|r|}
\hline
 & & & \multicolumn{3}{c|}{CPU time} & \multicolumn{3}{c|}{RMSD} \\
 & $\rho$ & $\sigma$ & BCD & SFSDP & NLP-FD & BCD & SFSDP & NLP-FD \\
\hline
$m=1000,$ & 0.1 & 0 & \textbf{4.1} & 4.2 & 14.6 & 3.13e-02 & \textbf{1.18e-05} & 1.27e-02 \\
\cline{3-9}
$n=100$ & & 0.1 & \textbf{4.2} & 13.7 & 20.1 & 3.38e-02 & 1.92e-02 & \textbf{1.37e-02} \\
\cline{3-9}
 & & 0.2 & \textbf{8.8} & 14.5 & 33.6 & 3.74e-02 & 2.94e-02 & \textbf{1.87e-02} \\
\hline
$m=3000,$ & 0.1 & 0 & 12.8 & \textbf{12.7} & 44.0 & 4.02e-04 & \textbf{4.64e-07} & 7.87e-04 \\
\cline{3-9}
$n=300$ & & 0.1 & \textbf{1.7} & 77.7 & 81.9 & \textbf{2.78e-03} & 8.78e-03 & 3.18e-03 \\
\cline{3-9}
 & & 0.2 & \textbf{2.0} & 63.8 & 120.3 & 6.96e-03 & 1.72e-02 & \textbf{6.02e-03} \\
\cline{2-9}
 & $\sqrt{10/m}$ & 0 & \textbf{11.2} & 15.3 & 38.0 & 1.28e-02 & \textbf{1.51e-05} & 4.03e-03 \\
\cline{3-9}
 & & 0.1 & \textbf{4.9} & 85.9 & 49.6 & 1.53e-02 & 8.47e-03 & \textbf{5.07e-03} \\
\cline{3-9}
 & & 0.2 & \textbf{4.1} & 91.3 & 52.5 & 1.83e-02 & 1.39e-02 & \textbf{7.03e-03} \\
\hline
$m=5000,$ & 0.1 & 0 & \textbf{18.3} & 23.2 & 54.3 & \textbf{8.22e-08} & 2.00e-07 & 4.75e-04 \\
\cline{3-9}
$n=500$ & & 0.1 & \textbf{2.5} & 184.8 & 162.0 & \textbf{1.90e-03} & 8.26e-03 & 2.17e-03 \\
\cline{3-9}
 & & 0.2 & \textbf{2.2} & 135.6 & 223.6 & \textbf{3.74e-03} & 1.62e-02 & 4.19e-03 \\
\cline{2-9}
 & $\sqrt{10/m}$ & 0 & \textbf{22.0} & 30.2 & 57.6 & 9.01e-03 & \textbf{5.26e-05} & 5.81e-03 \\
\cline{3-9}
 & & 0.1 & \textbf{11.4} & 198.7 & 70.4 & 1.17e-02 & 6.56e-03 & \textbf{5.29e-03} \\
\cline{3-9}
 & & 0.2 & \textbf{9.4} & 160.0 & 87.8 & 1.43e-02 & 1.04e-02 & \textbf{5.96e-03} \\
\hline
$m=20000,$ & 0.1 & 0 & \textbf{82.7} & 176.2 & 718.2 & 5.81e-06 & \textbf{6.50e-07} & 3.23e-04 \\
\cline{3-9}
$n=2000$ & & 0.1 & \textbf{29.2} & 643.5 & 1300.8 & \textbf{9.14e-04} & 8.33e-03 & 9.54e-04 \\
\cline{3-9}
 & & 0.2 & \textbf{26.0} & 544.9 & 1450.2 & 2.22e-03 & 1.76e-01 & \textbf{1.79e-03} \\
\cline{2-9}
 & $\sqrt{10/m}$ & 0 & 223.0 & 357.9 & \textbf{213.8} & 3.35e-03 & \textbf{6.73e-05} & 2.25e-03 \\
\cline{3-9}
 & & 0.1 & \textbf{129.2} & 1527.9 & 215.2 & 4.16e-03 & 3.10e-03 & \textbf{2.30e-03} \\
\cline{3-9}
 & & 0.2 & \textbf{110.6} & 1558.2 & 240.5 & 5.47e-03 & 4.86e-03 & \textbf{2.63e-03}\\
\hline
\end{tabular}}\label{tab:Experiment7_2dim}
\end{table}

First, the results for $d = 2$ are given in Table~\ref{tab:Experiment7_2dim},\footnote{For $m=1000$, only the case of $\rho=0.1$ is shown because $\sqrt{10/m}=0.1$.} wherein the proposed method is referred to as ``BCD." We can see that when the measured distances included no errors ($\sigma = 0$), the estimation time of the proposed method was the lowest in most cases; furthermore, even when the proposed method was not the fastest, its estimation time was almost the same as that of the fastest method. In terms of the estimation accuracy, SFSDP estimated the sensor positions with the best accuracy of all the methods, by an order of magnitude. However, comparing the proposed method and NLP-FD shows that there was no appreciable difference between their estimation accuracies. When the measured distances included errors ($\sigma = 0.1$ and 0.2), the proposed method estimated the sensor positions the most rapidly of all the methods, by an order of magnitude in all cases, and the estimation accuracy was about the same as those of the other two methods.
\begin{table}[tb]
\centering
\caption{Results of numerical experiments for sensors and anchors placed randomly in $[0, 1]^3$. For each parameter combination, the lowest CPU time and shortest RMSD are in bold.}
\scalebox{0.89}{\begin{tabular}{|l|l|l||r|r|r|r|r|r|}
\hline
 & & & \multicolumn{3}{c|}{CPU time} & \multicolumn{3}{c|}{RMSD} \\
 & $\rho$ & $\sigma$ & BCD & SFSDP & NLP-FD & BCD & SFSDP & NLP-FD \\
\hline
$m=1000,$ & 0.1 & 0 & \textbf{5.3} & 11.5 & 34.6 & 8.19e-03 & \textbf{2.79e-05} & 7.62e-03 \\
\cline{3-9}
$n=100$ & & 0.1 & \textbf{18.5} & 25.8 & 32.5 & 3.21e-02 & 5.50e-02 & \textbf{2.77e-02} \\
\cline{3-9}
 & & 0.2 & 29.4 & \textbf{25.4} & 42.4 & \textbf{4.71e-02} & 8.18e-02 & 5.28e-02 \\
\hline
$m=3000,$ & 0.1 & 0 & \textbf{8.6} & 24.7 & 62.3 & 3.97e-05 & \textbf{1.22e-06} & 1.06e-03 \\
\cline{3-9}
$n=300$ & & 0.1 & \textbf{3.3} & 100.6 & 132.0 & \textbf{7.01e-03} & 3.11e-02 & 1.09e-02 \\
\cline{3-9}
 & & 0.2 & \textbf{3.6} & 102.0 & 194.2 & \textbf{1.47e-02} & 5.79e-02 & 2.45e-02 \\
\cline{2-9}
 & $\sqrt[3]{15/m}$ & 0 & \textbf{35.6} & 61.6 & 83.1 & 2.12e-02 & \textbf{6.25e-05} & 4.26e-03 \\
\cline{3-9}
 & & 0.1 & \textbf{26.8} & 185.7 & 102.5 & 3.29e-02 & 3.22e-02 & \textbf{1.48e-02} \\
\cline{3-9}
 & & 0.2 & \textbf{69.5} & 158.6 & 113.8 & 4.32e-02 & 5.12e-02 & \textbf{2.75e-02} \\
\hline
$m=5000,$ & 0.1 & 0 & \textbf{17.5} & 47.5 & 126.0 & 5.11e-05 & \textbf{4.41e-07} & 1.09e-03 \\
\cline{3-9}
$n=500$ & & 0.1 & \textbf{5.5} & 195.5 & 291.2 & \textbf{5.50e-03} & 2.94e-02 & 7.95e-03 \\
\cline{3-9}
 & & 0.2 & \textbf{24.0} & 170.3 & 284.9 & \textbf{1.19e-02} & 5.35e-02 & 1.38e-02 \\
\cline{2-9}
 & $\sqrt[3]{15/m}$ & 0 & \textbf{30.4} & 199.9 & 114.0 & 1.11e-02 & \textbf{2.13e-04} & 7.15e-03 \\
\cline{3-9}
 & & 0.1 & \textbf{11.9} & 554.8 & 182.8 & 2.43e-02 & 2.51e-02 & \textbf{1.18e-02} \\
\cline{3-9}
 & & 0.2 & \textbf{66.5} & 522.3 & 164.7 & 3.76e-02 & 4.04e-02 & \textbf{2.08e-02} \\
\hline
$m=20000,$ & 0.1 & 0 & \textbf{128.9} & 288.6 & 1813.7 & 5.53e-05 & \textbf{4.24e-07} & 6.26e-04 \\
\cline{3-9}
$n=2000$ & & 0.1 & \textbf{56.9} & 1271.6 & 2551.9 & \textbf{2.81e-03} & 2.93e-02 & 3.37e-03 \\
\cline{3-9}
 & & 0.2 & \textbf{56.9} & 1028.0 & 2679.0 & 7.59e-03 & 2.43e-01 & \textbf{5.61e-03} \\
\cline{2-9}
 & $\sqrt[3]{15/m}$ & 0 & \textbf{304.4} & OOM & 341.3 & \textbf{3.93e-03} & OOM & 4.00e-03 \\
\cline{3-9}
 & & 0.1 & \textbf{149.9} & OOM & 606.7 & 1.18e-02 & OOM & \textbf{6.08e-03} \\
\cline{3-9}
 & & 0.2 & \textbf{107.5} & OOM & 628.1 & 2.01e-02 & OOM & \textbf{1.06e-02}\\
\hline
\end{tabular}}\label{tab:Experiment7_3dim}
\end{table}

Next, the results for $d = 3$ are given in Table~\ref{tab:Experiment7_3dim}.\footnote{For $m=1000$, the case of $\rho=\sqrt[3]{15/m}$ is omitted because $\sqrt[3]{15/m}=0.24\approx 0.25$. An entry of ``OOM" means that we could not estimate the sensor positions because of insufficient memory (i.e., ``out of memory'').} The results for the three-dimensional scenario were similar to those for the two-dimensional scenario; that is, when the measured distances did not include errors, the estimation time of the proposed method was the lowest in each case. In terms of the estimation accuracy, SFSDP estimated the sensor positions with the highest accuracy of all the methods, by an order of magnitude; however, comparing the proposed method and NLP-FD again shows that there was no appreciable difference between their estimation accuracies. When the measured distances included errors, the estimation time of the proposed method was the lowest in all cases except $(m, n, \rho, \sigma) = (1000, 100, 0.1, 0.2)$, and even in that case, its estimation time was also almost the same as that of the fastest method (SFSDP). The estimation accuracy of the proposed method was also comparable to those of the other two methods. In addition, for $m=20000$, SFSDP could not estimate the sensor positions because of insufficient memory, while the proposed method could estimate the positions without running out of memory.

Overall, these results for the two- and three-dimensional cases show that the proposed method has practical advantages over the other methods: it can estimate sensor positions faster than those methods can without sacrificing the estimation accuracy, especially when measurement errors are included, and it does not run out of memory even for large-scale SNL problems.

It is interesting to consider why the proposed method and NLP-FD, both of which account for the rank constraint, could not estimate the sensor positions with as much accuracy as SFSDP when there were no measurement errors. The reason was probably because a formulation that accounts for the rank constraint is a nonconvex optimization problem and thus might converge to a stationary point that is not a global optimal solution. In contrast, if a problem is uniquely localizable, then SFSDP can estimate accurate sensor positions because the convergence to the global optimum of a relaxation problem is guaranteed. On the other hand, when measurement errors are included, even if a global minimum solution of the objective function $f$ is obtained and the optimal value is zero, it does not mean that the true sensor positions are estimated, but rather that the positions are estimated incorrectly. In other words, even if the objective function is strictly minimized, it does not necessarily mean that a good estimate of the sensor positions is obtained; thus, when measurement errors are included, estimation accuracy comparable to that of SFSDP can be obtained even with methods that account for the rank constraint and may cause the generated sequence to fall into a stationary point that is not a global optimal solution.

\section{Conclusion}
\label{sec:conclusion}
In this paper, we proposed a new method that transforms the formulation of problem (\ref{prob:rankSDP_error}), which appears in SNL, into an unconstrained multiconvex optimization problem~(\ref{prob:biconvex2}), to which the block coordinate descent method is applied. We also presented theoretical analyses of the proposed method. First, we showed that each subproblem that appears in Algorithm~\ref{alg:node-based} can be solved analytically. In addition, we showed that any accumulation point $(U^*,V^*)$ of the sequence $\{(U^{(p)},V^{(p)})\}_{p=1}^\infty$ generated by the proposed algorithm is a stationary point of the objective function of problem~(\ref{prob:biconvex2}), and we gave a range of $\gamma$ such that $(U^*,V^*)$ satisfies $U^*=V^*$.
We also pointed out the relationship between the objective function of problem~(\ref{prob:biconvex2}) and the augmented Lagrangian.
Numerical experiments showed that our method does inherit the rank constraint and that it can estimate sensor positions faster than other methods without sacrificing the estimation accuracy, especially when the measured distances contain errors, and without running out of memory.

The present study suggests three directions for future work. First, Algorithm~\ref{alg:node-based} uses a cycle rule in which the $2m$ subproblems are solved in the order of $\bm{u}_1,\dots,\bm{u}_m,\bm{v}_1,\dots,\bm{v}_m$. However, in the general coordinate descent method, there are other update rules such as a random rule and a greedy rule \cite{Luo1992,Wright2015}. In SNL, the strategy of updating from variables corresponding to sensors that are connected directly to anchors is also expected to improve the estimation accuracy and time. Therefore, there is still room to consider how the order of solving the $2m$ subproblems affects the estimation time and accuracy. Second, we performed the minimization sequentially with respect to each column of $U$ and $V$ for computational efficiency, but updating some columns of $U$ and $V$ together is also possible, and the manner of block division in applying the block coordinate descent method should be examined further. Finally, the proposed method could be extended to general quadratic SDP problems with a rank constraint.

\section{Acknowledgment}
We thank Dr.\ Xiaokai Chang for providing the NLP-FD MATLAB code. This is a preprint of an article published in Optimization Letters. The final authenticated version is available online at:\\
https://doi.org/10.1007/s11590-021-01762-9.

\section{Funding}
This work was supported by JSPS KAKENHI Grant Number JP20H02385.

\section{Conflict of interest}
The authors declare that they have no conflict of interest.

\end{document}